\documentclass{amsart}
\usepackage{stackrel}
\usepackage[utf8]{inputenc}

\usepackage[english] {babel}
\usepackage{amsmath,amssymb,amscd}
\usepackage{amsthm}
\usepackage{verbatim}
\usepackage{ifthen}
\usepackage[all]{xy}
\usepackage{yhmath}

\usepackage{bookmark,hyperref}
\hypersetup{
     colorlinks=true,
     linkcolor=blue,
     filecolor=blue,
     citecolor = blue,
     urlcolor=blue,
     }

\newcommand{\R}{{\mathbb R}}

\newcommand{\I}{{\mathcal{O}}}

\newtheorem{theorem}{Theorem}[section]
\newtheorem{corollary}[theorem]{Corollary}

\newtheorem{lemma}[theorem]{Lemma}
\newtheorem{proposition}[theorem]{Proposition}

\newtheorem{definition}[theorem]{Definition}
\newtheorem{remark}{Remark}

\emergencystretch=\maxdimen
\begin{document}

\title[Loops, Holonomy and Signature]{Loops, Holonomy and Signature}

\author{Juan Alonso}
\address{Universidad de la Rep\'ublica, Centro de Matem\'atica,  Facultad de Ciencias, Igu\'a 4225, 11400 Montevideo
\\Uruguay.}
\email{juan@cmat.edu.uy}

\author{Juan Manuel Burgos}
\address{Instituto de Matemática y Estadística, FING, UDELAR, Av. Julio Herrera y Reissig 565, CP 11300, Montevideo, Uruguay\\ Unidad Académica Matemática, FCEA, UDELAR, Gonzalo Ramírez 1926, CP 11200, Montevideo, Uruguay.}
\email{jmburgos@fing.edu.uy}

\author{Miguel Paternain}
\address{Universidad de la Rep\'ublica, Centro de Matem\'atica,  Facultad de Ciencias, Igu\'a 4225, 11400 Montevideo
\\Uruguay.}
\email{miguel@cmat.edu.uy}

\begin{abstract}
We show that there is a topology on certain groups of loops in Euclidean space such that these groups are embedded in a Fréchet-Lie group which is the structural group of a principal bundle with connection whose holonomy coincides with the Chen signature map. We also give an alternative geometric new proof of the Chen signature theorem and a generalization of this theorem in classes strictly containing the one originally considered by Chen.
\end{abstract}



\maketitle

\section{Introduction}

In this paper we consider the  relationship between holonomy maps of principal bundles with a connection and the signature map (see \eqref{Chen_map}) defined by Chen (\cite{Chen1},  \cite{Chen2prima}, \cite{Chen2}).

If  $\gamma$ is a piecewise regular curve in $\R^n$, Chen  defines the formal power series
\[
\Theta(\gamma)\, =\,1\,+\, \sum^{\infty}_{p=1}\,\  \sum_{i_1,\,i_2,\, \ldots,\,i_p}\  \,\int_{\gamma} dx_{i_1}dx_{i_2} \ldots dx_{i_p}\        X_{i_1}X_{i_2} \ldots X_{i_p}
\]
in the noncommutative indeterminates $X_1, \ldots, X_n$, where the coefficients are {\em iterated integrals} defined in section \ref{signature_section}. This will be called the \textit{Chen signature map}, and it is a homomorphism with respect to concatenation of curves. The Chen signature theorem (See Theorem \ref{Chen_Signature_Theorem} below) states that if 
$\Theta(\gamma)=1$ then $\gamma$   is a retraceable loop i.e.  a trivial loop under the equivalence relation generated by the identifications   $\alpha aa^{-1}\beta\sim \alpha\beta$.   
More recently (see  \cite{Lyons1, Lyons2, Lyons2prima, Lyons3}) the signature map was generalized for curves with less regularity, and it was shown that the kernel of the signature map  on paths with bounded variation consists of tree-like paths (see section  \ref{section_Juan}).

Holonomy maps appear in  topology and differential geometry and  they trace back to the well known theory of representations of the fundamental group and constructions of local systems over a manifold by the holonomy map of a principal bundle with a flat connection. A natural question is how to extend this construction to arbitrary connections. This question leads directly to the loop space of the manifold, an object that will be defined shortly \footnote{Do not confuse this space with the one defined in algebraic topology. The ambiguity in the names will disappear after we introduce the right notation.}. In contrast to the fundamental group of a manifold, the loop space we are referring to here will not be a purely topological object.
Specifically, while the fundamental group of a manifold coincides with that computed in the category of smooth curves, the loop space of the manifold, and therefore the results concerning this space, will depend on the category of curves where it is defined.

We show in Theorem \ref{main} that the Chen signature map can be regarded as the holonomy map of a suitable connection on a principal bundle. The description of this very natural and simple  geometric interpretation is one of the main aims of this note. The other one is to give a geometric proof  of the Chen signature theorem, in Theorem \ref{Chen_Signature_Theorem}, which is based on Theorem 1.7 of 
\cite{Holonomy}. This approach also allows us to directly  obtain a  major ingredient in the original proof of the Chen signature theorem, namely Lemma \ref{fundamental} whose content is  similar to the  Fundamental Lemma 3.5 in  \cite{Chen2}.  
Below we expand on these ideas in more detail, including the precise statements of the main results and some extensions, as well as a discussion of the different loop spaces we consider.  

Consider a point $p$ in a manifold $M$. We say that a class of loops based at $p$ is  \emph{concatenable} if the  concatenation of two loops in the class belongs to the class. Examples of classes with this property are the class of piecewise smooth loops $\Omega^{ps}(M,p)$,   the     
class of piecewise smoothly immersed loops $\Omega^{psi}(M,p)$ or the 
 class of piecewise analytic loops $\Omega^{pa}(M,p)$.      
By identifying loops of a concatenable class by reparameterization, we obtain associativity of the operation but in general $\alpha\,\alpha^{-1}\neq c$ where $c$ denotes the constant loop. A further equivalence relation is needed, and we consider the following three. Firstly, a natural relation is the one finitely generated by $\alpha aa^{-1}\beta\sim \alpha\beta$. This will be called the \emph{retrace relation}.  
The corresponding groups are denoted by  $\mathcal{L}^{\ ps,ret}(M,p)$, $\mathcal{L}^{\ psi ,ret}(M,p)$ or $\mathcal{L}^{\ pa,ret}(M,p)$ depending on the choice of the class of curves. 
Then there is the relation by {\em thin} or {\em zero area} homotopies, that we define in Section \ref{s.prelim}. We denote the resulting groups by $\mathcal{L}^{\ ps,thin}(M,p)$, $\mathcal{L}^{\ psi ,thin}(M,p)$ or $\mathcal{L}^{\ pa,thin}(M,p)$ respectively.

From the point of view of holonomy,  the natural relation is the one which identifies a pair of loops if, given a Lie group $G$, their holonomies coincide for every $G$-principal bundle with connection. This relation will be called the \emph{G-holonomy relation}. 
The   corresponding groups are  denoted by  $\mathcal{L}^{\ ps,G}(M,p)$, $\mathcal{L}^{\ psi ,G}(M,p)$ or $\mathcal{L}^{\ pa, G}(M,p)$ depending on the choice of the class of curves.


For this work it will be important to consider the difference between the groups obtained by these equivalence relations. We quote the known facts about this in Section \ref{s.prelim}, but we remark one result here: In \cite{Meneses}, Meneses  provides  a non-retraceable thin loop in the class of piecewise smoothly loops, showing in particular that the group $\mathcal{L}^{\ ps,ret}(M,p)$ is a proper extension of $\mathcal{L}^{\ ps,thin}(M,p)$. We show that no such counterexample is possible in the class of piecewise smoothly immersed loops.

 Actually, in section \ref{section_Juan} we prove the following proposition.

\begin{proposition}\label{Equivalencia_ret_thin}
In the class of piecewise smoothly immersed curves based at a point $p$ in $M$, the retrace and thin homotopy relations coincide. That is, there is a natural isomorphism
$$\mathcal{L}^{\ psi,ret}(M,p)\ \cong\ \mathcal{L}^{\ psi,thin}(M,p).$$
\end{proposition}

From now on, unless otherwise specified, we will denote the groups in the previous proposition simply by $\mathcal{L}(M, p)$ and by $\mathcal{L}_p$ in the case the manifold is the Euclidean space $\R^n$.

In our next result we construct a principal bundle whose holonomy map is the signature on loops. 
Let $\overline{G}_n$ be the Fréchet-Lie group corresponding to the completion of the real free Lie algebra generated by $X_1,\ldots, X_n$. (See section \ref{section_holonomy_signature} for more details and references). 
\begin{theorem}\label{main}
There is a $\overline{G}_n$-principal bundle with connection $(E,\,\R^n,\,\overline{G}_n,\,\pi,\,\theta)$ such that for any $p$ in $\R^n$, the holonomy map 
$$Hol_{\,\theta,\,p}:\,\mathcal{L}_p\,\hookrightarrow\,\overline{G}_n$$
is a monomorphism	 and it coincides with the Chen signature map  on $\mathcal{L}_p$. 
\end{theorem}

For this work we consider principal bundles with left actions of the structure group. This  is not the most common convention but it is what makes the holonomy a group homomorphism. When the action is on the right, the resulting holonomy reverses the order of the product.   Chen \cite{Chen2prima} also defines an alternative version of the signature $\Theta^{*}$, satisfying    
$\Theta^{*}(\alpha\beta)=\Theta^{*}(\beta)\Theta^{*}(\alpha)$. Taking the action on the right would lead to $\Theta^{*}$ as holonomy (see Remark \ref{leftaction}).  
One feature of our approach is that it provides simple geometric  characterizations of both signatures.

In \cite{Neretin}, the author introduces for $\xi>0$ a family of Banach-Lie groups $\mbox{Fr}_n^{\,\xi}\subset \overline{G}_n$ allowing us to  strengthen  the above result:

\begin{theorem}\label{main2}
For every loop $\gamma$ in $\mathcal{L}_p$, the element $Hol_{\,\theta,\,p}(\gamma)\,=\,\Theta(\gamma)$ belongs to the Banach-Lie group $\mbox{Fr}_n^{\,\xi}$ for every $\xi>0$.
\end{theorem}





A natural question is whether the induced topologies on the group of loops $\mathcal{L}_p$ from the Banach-Lie groups $\mbox{Fr}_n^{\,\xi}$ for $\xi>0$ coincide.

In section \ref{geometric_proof} we give a geometric alternative new proof of the Chen signature theorem, quoted as Theorem \ref{Chen_Signature_Theorem} here.  
Our proof is based on Theorem 1.7 of 
\cite{Holonomy} which states that  $\mathcal{L}^{\ psi,ret}(M,p)\ \cong\ \mathcal{L}^{\ psi,G}(M,p)$ for  connected and non-solvable $G$.  This approach allows us to obtain directly a result (Lemma \ref{fundamental}) whose content is  similar to the Fundamental Lemma 3.5 in  \cite{Chen2}. 
 The proof of Theorem \ref{Chen_Signature_Theorem}  also shows that the Chen signature theorem also holds 
for piecewise smooth curves, replacing the result in  \cite{Holonomy} by a result in 
 \cite{Tlas} giving that $\mathcal{L}^{\ ps,thin}(M,p)\ \cong\ \mathcal{L}^{\ ps,G}(M,p)$ for $G$ semisimple.   This stronger version of the signature theorem implies in turn the following theorem.

\begin{theorem}\label{generalization}
Consider the formal power series ring $R\,=\,\R[[X_1,\ldots\,X_n]]$ with the non-commutative variables $X_1,\ldots\,X_n$. Then, the Chen map \eqref{Chen_map} is a faithful representation in the ring $R$ of the group of loops in the class of piecewise smooth loops under thin homotopy relation $\mathcal{L}^{\ ps,thin}(\R^n,p)$.
\end{theorem}

Identifying the group of loops with its image under the holonomy map in Theorem \ref{main}, we have another corollary of the theorem.

\begin{corollary}\label{Lie_group_generated}
Consider $p$ in $\R^n$. For every $\xi>0$, there is a topology $\tau$ on $\mathcal{L}_p$ and a proper closed normal Banach-Lie subgroup $H$ of $\mbox{Fr}_n^{\,\xi}$ verifying the following:
\begin{enumerate}
\item $H$ contains an embedded copy of the space $(\mathcal{L}_p,\,\tau)$ as a topological subgroup; i.e. $\mathcal{L}_p\subset H$.
\item It is \textit{simple relative to} $\mathcal{L}_p$: there are no proper closed normal subgroups $H'\lhd H$ such that $\mathcal{L}_p\subset H'$.
\item The Lie algebra of $H$ contains an infinitely generated free Lie algebra.
\end{enumerate}
\end{corollary}

As an immediate corollary of the previous theorem we have the following complement to Theorem \ref{main}.

\begin{corollary}
Theorem \ref{main} and its corollaries hold also for the group of loops in the class of piecewise smooth loops under thin homotopy relation $\mathcal{L}^{\ ps,thin}(\R^n,p)$.
\end{corollary}

\section{Preliminaries}\label{s.prelim}

\subsection{Loops and Holonomy}

Let $I$ be the unit interval and $M$ a smooth manifold. We begin by recalling some standard notations. A {\em path} in $M$ is a piecewise smooth function from $I$ to $M$, and we say that two paths $a,\,b:\,I \rightarrow\, M$ are equivalent modulo
reparametrization if there is an orientation preserving piecewise smooth homeomorphism $\sigma\,:I\rightarrow I$ with piecewise smooth inverse such that $a\circ \sigma\,=\,b$.  
Denote by $\I_0(M)$ the quotient set under this equivalence relation. If $a(1)\,=\,b(0)$ we define $ab$ and $a^{-1}$ as follows: $ab\,(t)=a\,(2t)$ if $t\in [0,1/2]$ and $ab\,(t)=b\,(2t-1)$ if $t\in [1/2,1]$; $a^{-1}(t) = a\,(1-t)$ for every $t\in [0,1]$. Let
$e_p \in \I_0(M)$ be the constant path at the point $p$, i.e. $e_p(t) =p$ for every $t\in [0,1]$.

We need to consider another preliminary equivalence, which amounts to collapse constant sub-paths. Let  $a$ be a non-constant path in $M$. We shall define a {\em minimal form} $a_r$ for $a$ as follows: let $I_i \subset I$ be the family of maximal subintervals in which
$a$ is constant, and let $\sigma: I \rightarrow I$ be a surjective non-decreasing piecewise smooth function, constant in each $I_i$ and strictly
increasing in $I - \bigcup_i\, I_i$. Then there is $a_r: I \rightarrow M$ such that $a\, =\,a_r \circ\sigma$, which is non-constant on any subinterval of $I$ (this map is obtained by a universal property of quotients). Different choices of the function $\sigma$ give rise to minimal forms that are equivalent modulo reparametrization, and moreover, if two paths $a$ and $b$ are equivalent, so are any of their minimal forms $a_r$ and $b_r$. This allows us to define the {\em minimal class} of an element of $\I_0(M)$ (as the class of any minimal form of any representative), and take a quotient $\I_1(M)$ where we identify two elements of $\I_0(M)$ if they have the same minimal class (extending the definition to constant paths in the trivial way). The product and inverse are well defined on $\I_1(M)$, and the classes of constant paths are units for the product.

Let $\I^{psi}(M)\subset \I_1(M)$ be the set of classes of either constant paths or paths that are {\em piecewise smoothly immersed}, i.e. a finite concatenation of smooth immersions.
Notice that for $\alpha\in \I^{psi}(M)$ there are well defined notions of endpoints $\alpha(0)$ and $\alpha(1)$. Throughout the paper we will abuse of notation and refer to the elements $\alpha \in \I^{psi}(M)$ also as {\em curves}, and say that $\alpha$ is a {\em closed curve} if $\alpha(0)=\alpha(1)$.

In the set $\I^{psi}(M)$, consider the equivalence
relation finitely generated by the identifications $\alpha aa^{-1}\beta\sim\alpha\beta$. This is called the \emph{retrace relation}. With the formal definition in hand, we recall the concepts from the introduction: Let ${\mathcal E}^{\,psi,ret}(M)$ denote the quotient set of $\I^{psi}(M)$ under retrace relation, and let ${\mathcal L}^{\ psi,ret}(M,p)\subset {\mathcal E}^{\,psi,ret}(M)$ be the projection under the quotient map of the set of closed curves starting and ending at the point $p$. Note that
${\mathcal L}^{\ psi,ret}(M,p)$ is a group under concatenation 
whose neutral element is the equivalence class of $e_p$,
the constant path at $p$.

Another equivalence relation in $\I^{psi}(M)$ is given by the finite composition of thin homotopies. A \emph{thin homotopy} between two curves $\alpha$ and $\gamma$ is a homotopy $\eta: [0,1]^2\rightarrow M$ such that its image is contained in the union of the images of the curves, that is
$$\eta\left([0,1]^2\right)\,\subseteq\,\alpha\left([0,1]\right)\,\cup\,\beta\left([0,1]\right).$$
Denote the quotient under this relation by ${\mathcal E}^{\,psi,ret}(M)$ and let $\mathcal{L}^{\ psi,thin}(M,p)$ be the classes based at $p$. Note that, as before, it is a group under concatenation 
whose neutral element is the equivalence class of $e_p$, the constant path at $p$.

We can make analogous definitions in the class of piecewise smooth paths, obtaining the spaces ${\mathcal E}^{\,ps,ret}(M)$, ${\mathcal E}^{\,ps,thin}(M)$, ${\mathcal L}^{\ ps,ret}(M,p)$ and $\mathcal{L}^{\ ps,thin}(M,p)$, respectively.

It turns out that in general the resulting groups under the retrace and $G$-holonomy relation are not isomorphic. However, as it was shown in \cite{Holonomy}, it is quite remarkable that for every connected and non-solvable Lie group $G$, these groups are isomophic in the class of piecewise analytic loops,
\begin{equation}\label{Spalla_analytic}
\mathcal{L}^{\ pa,ret}(M,p)\ \cong\ \mathcal{L}^{\ pa,G}(M,p),\qquad G\ \rm{connected\ and\ non-solvable.}
\end{equation}

In \cite{Tlas}, Tlas proves a similar result in the concatenable class of $C^1$ loops based at $p$ with zero derivative at $p$. He proves that if the Lie group $G$ is semisimple then a pair of loops in this class are $G$-holonomy related if and only if they are related by a thin homotopy. In the class of piecewise smooth loops, Tlas result is equivalent to the following isomorphism under \emph{thin homotopy relation} (see Remark $4$ in \cite{Meneses})
\begin{equation}\label{Tlas}
\mathcal{L}^{\ ps,thin}(M,p)\ \cong\ \mathcal{L}^{\ ps,G}(M,p),\qquad G\ \rm{semisimple.}
\end{equation}

Along the lines of this result, Meneses in \cite{Meneses} comments that the analog result of \eqref{Spalla_analytic} claimed in \cite{Holonomy} is not valid in the class of piecewise smooth loops and he gives a counterexample of a thin and non-retraceable loop, hence
\begin{equation}\label{Spalla_smooth}
\mathcal{L}^{\ ps,ret}(M,p)\ \not\cong\ \mathcal{L}^{\ ps,G}(M,p),\qquad G\ \rm{semisimple.}
\end{equation}

One of the difficulties with piecewise smooth loops, in contrast with piecewise analytic ones, is that they can intersect in very complicated ways. This led Baez and Sawin in \cite{BaezSawin} to consider instead the class of piecewise smoothly immersed loops $\Omega^{psi}(M,p)$. In this class they develop the technology of \emph{tassels} and \emph{webs}. A careful inspection into the work in \cite{Holonomy}, which uses the technology developed in \cite{BaezSawin}, shows that the claimed result holds for the class of piecewise smoothly immersed loops.

The following is Theorem 1.7 from \cite{Holonomy} and will be referred to as the \emph{holonomy theorem} throughout the paper.   \footnote{We do not know why Spallanzani did not write the piecewise immersion hypothesis.}

\begin{theorem}\label{spa}
If $G$ is a connected and non-solvable Lie group, then there is a natural isomorphism $\mathcal L^{\ psi, ret}(M,p)\,\cong\, \mathcal{L}^{\ psi, G}(M,p)$ for every point $p$ in $M$.
\end{theorem}

All of these loop spaces equivalences have in common some sort of factorization. In \cite{Tlas}, Tlas associates a transfinite word for every loop in the class described before and identifies a pair of loops if they have the same reduced word. A loop whose reduced word is trivial is called a \emph{whisker} by him. In \cite{Holonomy}, Spallanzani uses the technology of tassels and webs developed by Baez and Sawin \cite{BaezSawin} to factorize smoothly immersed loops. The factorization in the piecewise analytic class follows almost directly by definition.


Another natural question is, given a manifold, how to reconstruct the principal bundle with connection from its holonomy map. This led to an abstract formulation of holonomy both in mathematics \cite{Milnor, Lashof, Teleman1, Teleman2} and theoretical physics as well \cite{Barret, BaezSawin, Gambini, Loll}. We recommend the recent paper \cite{Meneses} by Meneses for references and historical account of the theory.


As it was mentioned at the introduction, in general the retrace and thin relations are not equivalent. However, in the class of piecewise smoothly immersed  closed curves considered in this paper, they are. In section \ref{section_Juan}, we will prove Proposition \ref{Equivalencia_ret_thin} which states the existence of a natural isomorphism
$${\mathcal L}^{\ psi,ret}(M,p)\,\cong\, \mathcal{L}^{\ psi,thin}(M,p).$$
In view of this isomorphism, we will simply denote these groups by $\mathcal{L}(M,p)$.  In the case where the manifold $M$ is the Euclidean space $\R^n$, the case considered in this paper, from now on we will simply denote by $\mathcal L_p$ the group $\mathcal{L}(\R^n, p)$.

\subsection{Iterated Integrals and Signature}\label{subsection_Signature}\label{signature_section}
  
If  $\omega_{1},\,  \omega_2,\, \ldots,\,  \omega_q$ are one forms in a manifold $M$ and $\gamma:I\rightarrow M$ is a piecewise regular path, define the iterated integral by the formula

\[   \int_{\gamma}  \omega_{1}  \omega_2 \ldots  \omega_q \,=\, \int_{t_1<\ldots<t_q}\,  \omega_{1}(\gamma'(t_1))\,   \omega_2(\gamma'(t_2))  \ldots  \omega_q(\gamma'(t_q))\,dt_1\ldots dt_q.\]

Recall from the introduction that if $\gamma$ is a piecewise regular curve $\gamma$, Chen \cite{Chen2} defines the formal power series

\begin{equation}\label{Chen_map}
\Theta(\gamma)\, =\,1\,+\, \sum^{\infty}_{p=1}\,\  \sum_{i_1,\,i_2,\, \ldots,\,i_p}\  \,\int_{\gamma} dx_{i_1}dx_{i_2} \ldots dx_{i_p}\        X_{i_1}X_{i_2} \ldots X_{i_p}
\end{equation}
in the noncommutative indeterminates $X_1, \ldots, X_n$. This will be called the \textit{Chen signature map}. Chen proved that this map is a faithful representation of the group of loops $\mathcal{L}^{\ psi,ret}(\R^n,p)$ \cite{Chen1, Chen2prima, Chen2}. For other results concerning iterated integrals, see \cite{Chen3, Chen4}. In recent years this theory had a resurgence in the context of stochastic processes, 
 bounded variation paths,  and rough paths \cite{Lyons1, Lyons2, Lyons2prima, Lyons3}.
 

The Chen signature theorem can be stated as follows.

\begin{theorem}\label{Chen_Signature_Theorem}
If $\Theta(\gamma)=1$, then $\gamma$ is trivial in ${\mathcal E}^{\,psi,ret}(\mathbb R^n)$. 
\end{theorem}

\section{Retrace and thin relation on the piecewise immersed class}\label{section_Juan}

Next we prove Proposition \ref{Equivalencia_ret_thin} which states the existence of a natural isomorphism
$$\mathcal{L}^{\ psi,ret}(M,p)\ \cong\ \mathcal{L}^{\ psi,thin}(M,p).$$
That is to say that for piecewise smoothly immersed curves, thin homotopy is the same as retrace equivalence. We will use the concept of {\em tree-like paths} of Hambly and Lyons \cite{Lyons2prima}, in the equivalent formulation introduced in \cite{Lyons2} and \cite{Lyons4}:

\begin{definition} \label{tree-like} A path $\gamma:[a,b]\to M$ is tree-like if there is an $\R$-tree $T$ and continuous functions $\phi:[a,b]\to T$ and $\psi:T\to M$ such that $\gamma=\psi\circ\phi$ and $\phi(a)=\phi(b)$.
\end{definition} 

Namely, a tree-like path is a loop that factors through a loop in an $\R$-tree. In \cite{Levy}, L\'evy shows that for bounded variation curves, being a tree-like path is equivalent to being thin homotopic to the constant path. Also, for $C^1$ loops with zero derivative at the endpoints, Tlas \cite{Tlas} shows directly the equivalence between Definition \ref{tree-like} and being thinly homotopic to a constant path. Note that a piecewise smoothly immersed path is of bounded variation, and also can be reparametrized to be $C^1$ with zero derivative at its endpoints, so these results apply to them. 

Proposition \ref{Equivalencia_ret_thin} follows immediately from the previous remarks and the following lemma:

\begin{lemma} Let $\gamma:[a,b]\to M$ be a tree-like path that is piecewise smoothly immersed. Then $\gamma$ is retraceable (i.e. retrace equivalent to a constant path). 
\end{lemma}

\begin{proof}
Let $a=t_0<t_1<\cdots<t_k=b$ be a partition so that $\gamma|_{[t_i,t_{i+1}]}$ is a smooth embedding for $i=0,\ldots,k-1$. Consider the $\R$-tree $T$ and the maps $\phi:[a,b]\to T$ and $\psi:T\to M$ as in Definition \ref{tree-like}. Notice that $\phi|_{[t_i,t_{i+1}]}$ must be injective, so its image is the unique geodesic in $T$ with endpoints $\phi(t_i)$ and $\phi(t_{i+1})$. Then the image of $\phi$ is the convex hull of finitely many points, namely $\phi(t_i)$ for $i=0,\ldots,k$. In an $\R$-tree, the convex hull of finitely many points is (isometric to) a finite simplicial tree (this fact can be obtained by straightforward induction on the number of points). Moreover, each injective path $\phi|_{[t_i,t_{i+1}]}$ can be reparametrized to be the geodesic between $\phi(t_i)$ and $\phi(t_{i+1})$. After such reparametrization, $\phi$ is a closed edge-path in a simplicial tree, which is retraceable, and then so is $\gamma=\psi\circ\phi$.   

\end{proof}

\section{Holonomy and Signature}\label{section_holonomy_signature}

In this section we construct a principal bundle with connection whose holonomy is a monomorphism and coincides with the Chen signature map on loops.

Let $\mathcal T_n$ be the real associative algebra generated by the elements $X_1,\,X_2,\,\ldots,\,X_n$ and denote by $\mathfrak{g}_n$ the real Lie algebra freely generated by these.
Let $\overline{\mathcal T}_n$ and  $\overline{\mathfrak g}_n\subset \overline{\mathcal T}_n$ be the corresponding graded completions. As observed in \cite{Neretin}, the algebra $\overline{\mathcal T}_n$  is a Fr\'echet space. Indeed, the seminorms are defined as follows,
$$\Vert x\Vert_{i_1\ldots i_p}\,=\, \vert c_{i_1\ldots i_p}\vert\qquad\mathrm{if}\qquad x\,=\,\sum_{i_1\ldots i_p}\, c_{i_1\ldots {i_p}}     X_{i_1}X_{i_2} \ldots X_{i_p}.$$

For any $x\in \overline{\mathcal T}_n$, define the exponential operator as follows
$$\exp(x)\,=\,\sum_{k=0}^{\infty}\,\frac{x^k}{k!}.$$
Denote by $\overline{G}_n$ the image of $\overline{\mathfrak g}_n$ under the exponential operator, that is
$$\overline{G}_n\,=\,\exp\left(\overline{\mathfrak g}_n\right)\,\subset\,\overline{\mathcal T}_n.$$

It is well known (\cite{Lyons1}, Theorem 2.1.1), (\cite{Neretin}, Theorem 1.2)  and \cite{Reutenauer}  that $\overline{G}_n$  is a Fr\'echet-Lie group.
Consider the trivial left principal bundle $E\,=\,\R^n\times \overline{G}_n$ and 
$\omega \in \Omega^1(\R^n,\,\overline{\mathfrak{g}}_n)$ given by 

$$\omega=-\,\sum_{i=1}^n\,X_i\,dx^i.$$
This form induces a connection one-form   $\theta\in  \Omega^1(E,\,\overline{\mathfrak{g}}_n)$ given by

\[\theta_{(p, g)}(v, X)=g\omega_p(v)g^{-1}+Xg^{-1}\]

where $p,v\in \R^{n}$,  $g\in \overline{G}_n$ and   $X\in \overline{\mathfrak{g}}_n$.  It is straightforward to check that this form is left invariant.

\begin{remark}\label{leftaction} The action of $\overline{G}_n$ on $E$ is on the left. Chen \cite{Chen2prima} also defines an alternative version of the signature $\Theta^{*}$, satisfying    
$\Theta^{*}(\alpha\beta)=\Theta^{*}(\beta)\Theta^{*}(\alpha)$. Taking the action on the right would lead to $\Theta^{*}$ as holonomy.

\end{remark}
\medskip


\begin{proof}[Proof of Theorem \ref{main}]
We prove that the holonomy $Hol_{\,\theta,\,p}$ coincides with the Chen signature map \eqref{Chen_map} on loops.

By Theorem 2.1.2 in \cite{Lyons1}, the signature belongs to $\overline{G}_n$.  We have to show that the signature is the holonomy of the connection form $\theta$. Recall that 
$(\gamma, u)$ is a horizontal curve in $E$ iff for each $t$ we have $\theta(\gamma '(t), u'(t))=0$, i.e.
\begin{equation}\label{linear}
u'(t)= -u(t)\omega (\gamma'(t)).
\end{equation}
If $u$ is the solution satisfying $u(0)=1$ we obtain 

\begin{equation}\label{paso1}
 u(t)=1+\sum_{i_1}\int_{t_1<t}u(t_1)dx_{i_1}(\gamma'(t_1)) dt_1 \; X_{i_1} . 
\end{equation}

We write $u(t_1)$ according to \eqref{paso1}  integrating in $t_2$ which gives 

\begin{equation}
\begin{split}
 u(t)&=1+\sum_{i_1}\int_{t_1<t}dx_{i_1}(\gamma'(t_1)) dt_1 \; X_{i_1} + \\
  &+\sum_{i_2 i_1}\int_{t_2<t_1<t}u(t_2)dx_{i_2}(\gamma'(t_2))dx_{i_1}(\gamma'(t_1)) dt_2 dt_1 \; X_{i_2}X_{i_1}. 
\end{split}
\end{equation}

Repeating this procedure (see also \cite{Loll},  \cite{Neretin}) we obtain the series  

  coinciding with the signature. Because of the Chen signature Theorem \ref{Chen_Signature_Theorem}, the holonomy $Hol_{\,\theta,\,p}$ is a monomorphism and the result follows.
\end{proof}

This procedure also gives that the signature is a group element.  

\medskip

\begin{proof}[Proof of Theorem \ref{main2}]
The following definitions are taken from \cite{Neretin}. If $\xi>0$ and
\[x\,=\,\sum_{i_1\ldots i_p}\, c_{i_1\ldots {i_p}}\,     X_{i_1}X_{i_2} \ldots X_{i_p} \] 
the $\Vert x\Vert_{\xi}$ norm is defined by the expression
\[\Vert x\Vert_{\xi}\,=\, \sum_{i_1\ldots i_p}\,  {\xi} ^{p}   \vert c_{i_1\ldots {i_p}}\vert. \]

The Banach space ${\mathcal T}_n^{\xi}$ is defined as the linear subspace of those $x\in \overline{\mathcal T}_n$ verifying $\Vert x\Vert_{\xi}<\infty$ together with the norm above. Then set 
\[\mbox{ Fr}_n^{\,\xi}\,=\, \overline{ G}_n\cap {\mathcal T}_n^{\xi}.\]

Now notice that for every real $t$, we have $-\omega(\gamma'(t))\in  {\mathcal T}_n^{\xi} $. Therefore, the first part of Lemma 2.5 in \cite{Neretin} gives that the solution $u(t)$ of equation \eqref{linear} belongs to the group $\mbox{Fr}_n^{\,\xi}$ and we have the result.
\end{proof}
\medskip

\begin{proof}[Proof of Corollary \ref{Lie_group_generated}]
By Theorem \ref{main2}, the holonomy is a monomorphism and we can take the initial topology on $\mathcal{L}_p$. This gives an embedding $\mathcal{L}_p\subset \mbox{ Fr}_n^{\,\xi}$ as a topological group.

We can think about the smallest closed normal subgroup $H\unlhd \mbox{ Fr}_n^{\,\xi}$ such that $\mathcal{L}_p\subset H$. This subgroup exists by a standard Zorn lemma argument. By definition, it is \textit{simple relative to} $\mathcal{L}_p$. This proves the first and second items.

The Lie subgroup $H\unlhd \mbox{ Fr}_n^{\,\xi}$ is normal hence its Lie subalgebra $\mathfrak{h}\unlhd {\mathfrak{fr}}_n^{\,\xi}$ is actually an ideal in the Lie algebra ${\mathfrak{fr}}_n^{\,\xi} = \, \overline{\mathfrak{g}}_n\cap {\mathcal T}_n^{\xi}$ of the Lie group  $\mbox{ Fr}_n^{\,\xi}$ (see \cite{Neretin}). It is clear that $\mathfrak{h}$ is not trivial since otherwise the Lie group $H$ would be trivial as well, but this is absurd since it contains the non-trivial loop group $\mathcal{L}_p$. Since $\mathfrak{g}_n$ is a free Lie algebra and $\mathfrak{h}$ is a non-trivial ideal of ${\mathfrak{fr}}_n^{\,\xi}$, then $\mathfrak{g}_n\cap \mathfrak{h}$ is an infinitely generated free Lie algebra \cite{infinitely_generated} contained in $\mathfrak{h}$. This proves the third item and we have the result.
\end{proof}


As it was mentioned in the introduction, this raises the question of whether these initial topologies on $\mathcal{L}_p$, arising from the embeddings in $\mbox{ Fr}_n^{\,\xi}$ for each $\xi>0$, agree for all $\xi$, or at least stabilize when $\xi$ is large enough. 


\section{Geometric proof of the Chen Signature Theorem}\label{geometric_proof}

\subsection{Main lemma}

The content of the following Lemma is similar to the Fundamental Lemma 3.5 in  \cite{Chen2}. We show that it can be derived from Theorem \ref{spa} giving in the next subsection a geometric proof Theorem \ref{Chen_Signature_Theorem}.

\begin{lemma}\label{fundamental}
If $\gamma\in \Omega^{\,psi}(\mathbb R^n)$ is not trivial in ${\mathcal E}^{\,psi,ret}(\mathbb R^n)$, then there are compactly supported one forms $\omega_1,\ldots, \omega_q$ such that 
\[  \int_{\gamma}  \omega_{1}  \omega_2 \ldots  \omega_q\neq 0.\]
\end{lemma}
\begin{proof}

Let $\gamma$ be a non trivial class in ${\mathcal E}^{\,psi,ret}(\mathbb R^n)$ and set $x=\gamma(0)$. Assume first that $\gamma$ is closed; i.e. $\gamma$ is not trivial in $\mathcal L_x(\mathbb{R}^n )$. We consider the trivial bundle  $\mathbb{R}^n \times SL(2,\, \mathbb R)$. 

Since $SL(2,\, \mathbb R)$ is simple, then Theorem \ref{spa} applies, that is, there is a 
connection $A$ in the bundle $\mathbb{R}^n \times SL(2, \mathbb R)$, such that its holonomy $H_A(\gamma)$ is not trivial. This can also be deduced from Tlas result \cite{Tlas} and Proposition \ref{Equivalencia_ret_thin} as follows: since every whisker in the piecewise smooth class is a thin loop (see Remark $4$ in \cite{Meneses}), by Proposition \ref{Equivalencia_ret_thin} it is a retraceable loop in the piecewise smoothly immersed class. In particular, $\gamma$ is not a whisker and by Tlas result we have the same conclusion as before.

The holonomy can be expressed (see \cite{Loll}, p. 1416) as
\[  H_A(\gamma)= I+ \sum_{n>0}   (-1)^n \int_{t_1<\ldots< t_n}\,  A(\gamma'(t_1))\,   \ldots\,  A(\gamma'(t_n))\,  dt_1\ldots dt_n.   \]
Since  $H_A(\gamma)$ is not the identity, there is $q>0$ such that   
\[  \int_{t_1<\ldots< t_q} A(\gamma'(t_1))\,   \ldots\,  A(\gamma'(t_q))\,dt_1\ldots dt_q\neq 0.\]
Write $A=\left(a_{ij}\right)$ where $a_{i,j}$ are real valued one-forms. Note that the $(i,j)$- entry of 
\[  \int_{t_1<\ldots< t_q} A(\gamma'(t_1))\,   \ldots\,  A(\gamma'(t_q))\,dt_1\ldots dt_q\]
is given by 
\[  \int_{t_1<\ldots< t_q} \sum_{k_1,\ldots k_{q-1}}\, a_{ik_1}(\gamma'(t_1))\,   a_{k_1k_2}(\gamma'(t_2))\,  \ldots\,  a_{k_{q-1} j}(\gamma'(t_q))\,dt_1\ldots dt_q.\]
Then, there are indexes $i,\, j$ and a $(q-1)$-tuple  $k_1,\ldots, k_{q-1}$ such that 
\[  \int_{t_1<\ldots< t_q}  a_{ik_1}(\gamma'(t_1))\,   a_{k_1k_2}(\gamma'(t_2))\,  \ldots\,  a_{k_{q-1} j}(\gamma'(t_q))\,dt_1\ldots dt_q\neq 0,\]
i.e. there are $q$ real one-forms $\omega_1, \ldots, \omega_q$ such that 
\[  \int_{t_1<\ldots< t_q}\,  \omega_{1}(\gamma'(t_1))\,   \omega_2(\gamma'(t_2))\,  \ldots\,  \omega_q(\gamma'(t_q))\,dt_1\ldots dt_q\neq 0,\]
that is 
\[  \int_{\gamma}\,  \omega_{1}\,  \omega_2 \,\ldots\,  \omega_q\,\neq\, 0.\]

In the proof of Theorem 7.1 of  \cite{Holonomy} it is assumed that the support of $A$ is compact, hence the supports of $ \omega_{1},  \omega_2, \ldots  \omega_q$ are compact as well.
If $\gamma$ is not closed, i.e. $\gamma(0)\neq \gamma(1)$,  we consider a compactly supported real smooth function  $f$ such that $f(\gamma(0))\neq f(\gamma(1))$ and therefore obtain
$\int_{\gamma}df\neq 0$. 
\end{proof}

\subsection{Proof of Theorem \ref{Chen_Signature_Theorem}}

  The following definition and lemma,  which we include for the sake of completeness,  are taken  from \cite{Chen2}.

\begin{definition}
\[\int_a^b f_1(t)dt\ldots f_p(t)dt =\int_a^{b} \left(\int_a^{s}  f_1(t)dt\ldots f_{p-1}(t)dt\right) f_p(s) ds.  \]
\end{definition}

\begin{lemma}\label{real}
The product $\int_a^b\, f_1(t)dt\ldots f_i(t)dt\,  \int_a^b\, f_{i+1}(t)dt\ldots f_p(t)dt$ is a linear combination of integrals 
$\int_a^b\, f_{i_1}(t)dt\ldots f_{i_p}(t)dt$.
\end{lemma}

\begin{proof}
The cases $p=1$, $p=i$, $i=0$ are trivial so assume that $p\geq 2$ and $0<i<p$. Set 
\[g(s)=\int_a^s f_1(t)dt\ldots f_i(t)dt  \int_a^s f_{i+1}(t)dt\ldots f_p(t)dt,\]

\[h_1(s)=\int_a^s f_1(t)dt\ldots f_{i-1}(t)dt  \int_a^s f_{i+1}(t)dt\ldots f_p(t)dt,\]

\[h_2(s)=\int_a^s f_1(t)dt\ldots f_i(t)dt  \int_a^s f_{i+1}(t)dt\ldots f_{p-1}(t)dt.\]

Observe that
\[g(b)=\int_a^b g'(t)dt= \int_a^b h_1(t) f_i(t)dt  +   \int_a^b h_2(t) f_p(t)dt. \]
The Lemma is finished applying the induction hypothesis to $h_1$ and $h_2$.
\end{proof}

We say that integrals of the form   $\int_{\gamma}dx_{i_1}\ldots dx_{i_k}$ are elementary integrals.  From Lemma \ref{real} we obtain 
\begin{corollary}\label{elementary}
Any product of elementary integrals is a linear combination of elementary integrals.
\end{corollary}

Now, the kernel of the signature can be obtained as in \cite{Chen2}. 

Since the supports of $\omega_1\ldots\omega_q$ are compact, there are  $\overline\omega_1,\ldots,\overline\omega_q$ with polynomial coefficients such that 
\[  \int_{\gamma}\,  \overline\omega_{1}\,  \overline\omega_2\, \ldots \, \overline\omega_q\,\neq\, 0.\]
By additivity, we may assume that $\overline\omega_i$ have monomial coefficients, namely, there are monomials  $g_1,\ldots, g_q$  such that 
\[  \int_{\gamma}\,  g_{1}dx_{j_1}\,  g_2 dx_{j_2}\, \ldots\,  g_q dx_{j_q}\,\neq\, 0.\]
We have to show that this expression is a linear combination of elementary iterated integrals, i.e. integrals of the form
\[  \int_{\gamma}\,  dx_{i_1}  dx_{i_2} \ldots  dx_{i_p}, \]
and then one of those integrals has to be non zero, finishing the theorem. 

The claim follows by induction in $q$ as in Lemma 4.1 of \cite{Chen2}. Assume that $q=1$. Since $g_1$ is a monomial we have $g_1=x_1^{m_1}\ldots x_n^{m_n}$. Let $\gamma_t$ be the restriction of $\gamma$ to $[0, t]$. Hence we have 
\[x_i(\gamma(t))=\int_{\gamma_t}dx_i+ x_i(\gamma(0)),  \]
and thus  
\[  \int_{\gamma}  g_1 dx_{j_1}=\int^1_0   \left(\int_{\gamma_t}dx_1+ x_1(\gamma(0))\right)^{m_1}\ldots   \left(\int_{\gamma_t}dx_n+ x_n(\gamma(0))\right)^{m_n}\, dx_{j_1}(\gamma'(t)) dt \]
which gives, by Corollary \ref{elementary}, a linear combination of integrals 
\[ \int_0^{1}   \left(\int_{\gamma_t}dx_{i_1}dx_{i_2} \ldots dx_{i_k}\right)\,  dx_{j_1}(\gamma'(t)) dt= \int_{\gamma}dx_{i_1}dx_{i_2} \ldots dx_{i_k} dx_{j_1}.   \]

For the induction step, note that
\[  \int_{\gamma}  g_{1}dx_{j_1}  g_2 dx_{j_2} \ldots  g_q dx_{j_q} =\int_0^1 \left( \int_{\gamma_t}  g_{1}dx_{j_1}  g_2 dx_{j_2} \ldots  g_{q-1} dx_{j_{q-1}} \right) g_q(\gamma(t))dx_{j_q}(\gamma'(t))dt.\]
By the induction hypothesis,
\[\int_{\gamma_t}  g_{1}dx_{j_1}  g_2 dx_{j_2} \ldots  g_{q-1} dx_{j_{q-1}}\]
is a linear combination of integrals of the form
\[ \int_{\gamma_t}dx_{i_1}dx_{i_2} \ldots dx_{i_k}.  \]
Applying the same argument as in the base step we have that 
\[  \int_{\gamma}  g_{1}dx_{j_1}  g_2 dx_{j_2} \ldots  g_q dx_{j_q}\]
is a linear combination of integrals of the form
\[ \int_0^1 \left( \int_{\gamma_t}dx_{i_1}dx_{i_2} \ldots dx_{i_k} \int_{\gamma_t} dx_{i_{k+1}}dx_{i_{k+2}} \ldots dx_{i_{k+r}}\right)dx_{j_q}(\gamma'(t))dt. \]
which is a linear combination of elementary integrals by Corollary \ref{elementary}.

\section*{Acknowledgments}
We would like to thank Pablo Lessa for putting our attention on the Chen signature map and to Mauricio Velasco for his suggestions for improving the manuscript.


\begin{thebibliography}{Coh93}


\bibitem[Ba]{Barret}
J.W. Barrett, \emph{Holonomy and Path Structures in General Relativity and Yang-Mills Theory}, International Journal of Theoretical Physics \textbf{30} 9 (1991).


\bibitem[BGLY]{Lyons2}
H. Boedihardjo, X. Geng, T. Lyons, D. Yang, \emph{The signature of a rough path: uniqueness}, Adv. Math. \textbf{293} (2016), 720--737.


\bibitem[BS]{BaezSawin}
J. Baez, S. Sawin, \emph{Functional integration on spaces of connections}, Journal of Functional Analysis \textbf{150} (1997), 1--26.


\bibitem[Ch1]{Chen1}
K. T. Chen, \emph{Iterated integrals and exponential homomorphisms}, Proc. London Math. Soc. \textbf{4} (1954), 502--512.

\bibitem[Ch2]{Chen2prima}
K. T. Chen, \emph{Integration of Paths, Geometric Invariants and a Generalized Baker-Hausdorff Formula}, Annals of Mathematics, Second Series \textbf{65} 1 (1957), 163--178.

\bibitem[Ch3]{Chen2}
K. T. Chen, \emph{Integration of paths: a faithful representation of paths by non-commutative formal power series}, Trans. Amer. Math. Soc. \textbf{89} (1958), 395--407.

\bibitem[Ch4]{Chen3}
K. T. Chen, \emph{Iterated integrals of differential forms and loop space homology}, Ann. of Math. \textbf{97} 2 (1973), 217--246.

\bibitem[Ch5]{Chen4}
K. T. Chen, \emph{Iterated path integrals}, Bull. A.M.S. \textbf{83} 5 (1977).


\bibitem[Ek]{infinitely_generated}
N. Ekici, \emph{On ideals of free lie algebras}, Journal of Pure and Applied Algebra \textbf{45} 3 (1987), 201--205.


\bibitem[GT]{Gambini}
R. Gambini, A. Trias, \emph{Gauge Dynamics in the C-representation}, Nucl. Phys., \textbf{B278} (1986), 436--448.

\bibitem[HL]{Lyons2prima}
B. Hambly, T. Lyons. \emph{Uniqueness for the signature of a path of bounded variation and the reduced path group}, Ann. Math. \textbf{171} 1
(2010), 109--167.

\bibitem[HL2]{Lyons4}  B. Hambly, T. Lyons. \emph{Some notes on trees and paths}, arXiv:0809.1365, 2008.

\bibitem[La]{Lashof}
R. Lashof, \emph{Clasification of fibre bundles by the loop space of the base}, Ann. of Math. \textbf{64} (1956), 436--446.

\bibitem[Le]{Levy} T. L\'evy. \emph{The master field on the plane.} Asterisque (2017), 388.




\bibitem[Lo]{Loll}
R. Loll, \emph{Loop approaches to gauge field theories}, Theor. Math. Phys. \textbf{93} (1992), 1415--1432.

\bibitem[Ly]{Lyons1}
T. J. Lyons, \emph{Differential equations driven by rough signals}, Revista Matemática Iberoamericana \textbf{14} 2 (1998), 215--310.


\bibitem[LX]{Lyons3}
T. Lyons, W. Xu, \emph{Inverting the signature of a path}, Journal of the European Mathematical Society \textbf{7} 20 (2018), 1655--1687.


\bibitem[Me]{Meneses}
C. Meneses, \emph{Thin homotopy and the holonomy approach to gauge theories}, Contemp. Math. \textbf{775} (2021), 21pp.


\bibitem[Mi]{Milnor}
J. Milnor, \emph{Construction of universal bundles I}, Ann. of Math. \textbf{63} (1956), 272--284.




\bibitem[Ner]{Neretin}
Y. A. Neretin, \emph{A Lie group corresponding to the free Lie algebra and its universality}, arXiv:2411.11184.


\bibitem[Re]{Reutenauer}
C. Reutenauer \emph{Free Lie Algebras},  London Mathematical Society Monographs, New Series  Oxford Science Publications (1993). 

\bibitem[Sp]{Holonomy}
P. Spallanzani, \emph{Groups of loops and hoops}, Comm. Math. Phys. \textbf{216} 2 (2001), 243--253.


\bibitem[Te1]{Teleman1}
C. Teleman, \emph{Connections and bundles I}, Indagationes Mathematicae, \textbf{31} (1969), 89--103.

\bibitem[Te2]{Teleman2}
C. Teleman, \emph{Connections and bundles II}, Indagationes Mathematicae, \textbf{31} (1969), 104--112.

\bibitem[Tl]{Tlas}
T. Tlas, \emph{On the holonomic equivalence of two curves}, Internat. J. Math., \textbf{27} (2016), 1--22.




\end{thebibliography}
\end{document}